\documentclass[a4paper,10pt]{amsart}

\usepackage{amsmath,amsfonts,amsthm}
\newcommand{\R}{\mathbb{R}}		
				       
\newcommand{\Z}{\mathbb{Z}}						
\newcommand{\N}{\mathbb{N}}					
\newcommand{\Ti}{\mathcal{T}}                                       

\newtheorem{theo}{Theorem}
\newtheorem{rem}{Remark}

\title[Heat Kernel Bounds for ...]{Heat Kernel Bounds for the Laplacian on Metric Graphs of Polygonal Tilings}
\author{Ren\'e Pr\"opper}

\address{Institut f\"ur Analysis, Universit\"at Ulm, Helmholtzstrasse 18, 89081 Ulm, Germany}
\email{rene.proepper@uni-ulm.de}

\begin{document} 

\begin{abstract}
We obtain an upper heat kernel bound for the Laplacian on metric graphs arising as one skeletons of certain polygonal tilings of the plane, which reflects the one dimensional as well as the two dimensional nature of these graphs.

\end{abstract}
 \maketitle 

\section{Introduction}
If one looks at the heat semigroup on the metric graph with Kirchhoff type vertex conditions given by the uniform grid in $\R^2$, i.e. the vertex set is $\Z^2$ and the edges are all straight-line segments in $\R^2$ between vertices $(n,m)$ and $(n',m') \in \Z^2$ with $\|(n-n',m-m')\|_2=1$, then it is natural to guess that the semigroup behaves for small times like the one dimensional heat semigroup and for large times more like the two dimensional, i.e. the semigroup is governed by one and two dimensional Gaussian estimates. That this is indeed correct was shown in \cite{Pang}. The behaviour of the heat semigroup on compact metric graphs with more general vertex conditions was treated in \cite{Mu}.  \\
One would expect that the same result as for the uniform grid, which may be regarded as the boundary complex of the regular, edge-to-edge tiling of the plane with identical squares, is also true for the tiling with equilateral triangles or with regular hexagons or even more general tilings. Therefore, we have simplified the proof in \cite{Pang} in some places and have thus enlarged the scope of the method a bit. \\
The conditions we impose on the tilings are still rather restictive, but the tilings encompassed by our approach include some nice exemplars like triangulations of the plane, the 11 archimedean tilings, especially the three regular ones, and every other tiling composed from regular polygons of uniformly bounded area (from above and below); see \cite{GS} for an abundance of examples.\\
We want to mention that there is another well known approach (see e.g. \cite{Sal}) to getting upper as well as lower Gaussian estimates; this is via the volume doubling property and a global Poincar\'e inequality. In \cite{HLPP} this approach is shown to be feasible for large classes of metric graphs  through results in \cite{Haes} and especially in \cite{BBK}. The method we use in this paper is less generally applicable, but is in comparison more direct and seems to have some interest in its own right.

Now, we give a precise account of the setting we deal with.\\
Let $\Ti=(P_n)_{n \in \N}$ be a tiling of the (euclidean) plane by convex, compact polygons $P_n \subset \R^2$, $ n \in \N$, i.e. $\R^2  = \cup_{n \in \N} P_n$ and $P_n \cap P_m \subset \partial P_n$ for all $n,m \in \N$ with $n \neq m$, where $\partial P$ denotes the boundary of the polygon $P$. We do not require the tiling is edge-to-edge (see \cite{GS}). \\
The graph $G_\Ti=\cup_{n \in \N} \partial P_n$ of the polygonal tiling $\Ti$ is just its one-skeleton. $v \in G_\Ti$ is said to be a vertex of $G_\Ti$ if $v$ is a branching point of $G_\Ti$, i.e. for every neighbourhood $N_v$ of $v$ in $\R^2$ there exist at least three different $P \in \Ti$ having non empty intersection with $N_v$. We write $V_\Ti$ for the set of all vertices of $G_\Ti$; of course, $v \in V_\Ti$ iff $v$ is a vertex of at least  one polygon. The edges of $G_\Ti$ are given by the open, straight-line segments $e \subset G_\Ti$ such that there exist $v_1,v_2 \in V_\Ti$ with $ \{v_1,v_2\}= \overline{e} \cap V_\Ti$. In this case we say $v_1$ and $v_2$ are adjacent to each other and $v_1, v_2$ are incident to $e$, we write $v_1 \sim v_2$ and  $v_1 \sim e$, $v_2 \sim e$ resp. The set of all edges of $G_\Ti$ is denoted by $E_\Ti$ and the length of an edge $e \in E_\Ti$ by $l_e$. If convenient, we will identify the edges with open, real intervals $I_e=(0_e,l_e)$ and the vertices $v_1 \sim e$ and $v_2 \sim e$ with $0_e$ and $l_e$ resp.\\
We impose some further constraints on the tiling $\Ti$. Namely, that the size of the boundary of the polygons is uniformly bounded from above by a constant, $|\partial P|=\sum_{e \subset \partial P}l_e \leq M_\Ti < \infty$ for all $P \in \Ti$, that the edge lengths are uniformly bounded away from $0$, i.e. there exist $l_\Ti > 0$ such that $l_e \geq l_\Ti$ for every edge $e \in E_\Ti$, that every polygon $P \in \Ti$ posesses an incircle 
 with incentre $x_P$ and  inradius $r_P$, 
  and that the inradii are uniformly bounded from below, $0<h_{\Ti} \leq r_P \leq H_{\Ti} < \infty$ for all $P \in \Ti$, whereas the bound from above is due to $r_P < M_\Ti$. By $H_\Ti$, $h_\Ti$ and $M_\Ti$ we mean the optimal constants in each case.\\
Our assumptions about the tiling $\Ti$ also imply that the  graph $G_\Ti$  is connected and has uniformly bounded vertex degrees, $d_v=|\{e \in E_\Ti: e \sim v\}| \leq d_\Ti < \infty$ for all $v \in V_\Ti$, and that for every compact set $K \subset \R^2$ there are only finitely many polygons $P \in  \Ti$ with $P \cap K \not = \emptyset$. \\
The distance $d(a,b)  < \infty$ between two arbitrary points $a,b \in G_\Ti$ is defined as the length of a shortest path connecting $a$ and $b$ (inside $G_\Ti$).  Thus $G_\Ti$ is a complete path metric space.\\
As in the forthcoming investigations the tiling will always be fixed, we  mostly drop the subscript $\Ti$ in our notations. Moreover, for the rest of the introduction $G$ is allowed to be  any locally finite, metric graph ($d_v < \infty$ for every vertex $v$) with edge lenghts uniformly bounded away from $0$ and for simplicity also bounded from above. \\ 
We define the spaces $L^p(G)$, $1 \leq p \leq \infty$, in the usual way to consist of all Borel measurable functions $f:G \mapsto \R$ with 
$$ \int_{G} |f(x)|^p\, d\mu(x) < \infty, \quad 1 \leq p < \infty, \quad \hbox{and} \quad {\rm ess\,sup}|f| < \infty, \quad p=\infty,$$
where $\mu$ is the one dimensional Hausdorff measure.\\
$C_c^\infty(G)$ is the space of all continuous functions $f:G \mapsto \R$ with compact support and such that for every edge $e \in E$  the restriction of $f$ to $e$, denoted by  $f_{|e}$, is  in $C^\infty(\bar{e})$, where $C^\infty(\bar{e})$  is the space of all up to the boundary infinitely often differentiable functions on $e$.  We define the quadratic form $Q_0$ on $C_c^\infty(G)$ by
$$ Q_0(f,g):=\int_{G} f'(x)g'(x) d\mu(x)=\sum_{e \in E_\Ti}\int_0^{l_e} f'_{|e}(x)g'_{|e}(x) dx \qquad (f,g \in  C_c^\infty(G)).$$
The form $Q_0$ is closable in $L^2(G)$ with closure $Q$. The domain of $Q$ is 
$$D(Q)=\left \{f  \in L^2(G): \begin{array}{l} f \hbox{ is continuous on } G, \\ f_{|e} \in H^1(e) \hbox{ for all } e \in E, \\ f' \in L^2(G) \end{array} \right \}=:H^1(G).$$
When $H^1(G)$ is equipped with its natural norm $$\|f\|_{H^1(G)}:=\sqrt{\int_{G } f^2(x)+ (f'(x))^2\,d\mu(x)}, \qquad f \in H^1(G),$$ it becomes a densely and continuously imbedded subspace of $L^2(G)$. An appropriate cut-off and smoothing argument shows that $C_c^\infty(G)$ is indeed a core of $Q$.\\
The operator $A$ associated with the form $Q$ is the Laplacian with Kirchhoff type vertex conditions
\begin{align*} D(A)&=\left \{f \in H^1(G): \begin{array}{ll} f_{|e}' \in H^1(e) & \hbox{for all } e \in E,\\ f'' \in L^2(G), &   \\ \sum_{e \sim v}\frac{\partial f_{|e}}{\partial n}(v)=0  & \hbox{for all } v  \in V \end{array} \right \},\\ 
               A(f)&=-f'', 
\end{align*}  
where $\frac{\partial f_{|e}}{\partial n}(v)$ denotes the outer normal derivative of $f_{|e}$ at $v$, i.e. given $e=(0,l_e)$  we set $\frac{\partial f_{|e}}{\partial n}(v)=-f'_{|e}(0)$ if the vertex $v$ is identified with $0$ and $\frac{\partial f_{|e}}{\partial n}(v)=f'_{|e}(l_e)$ if $v$ is identified with $l_e$.\\
As $Q$ is a closed, symmetric, non-negative form which satisfies the Beurling-Deny criteria, $Q$ is a Dirichlet form and $-A$ generates a symmetric Markov semigroup $e^{-At}$ (see \cite{Da1}).

\section{Nash Inequalities}

In this section we establish two different Nash inequalities, an one dimensional and a two dimensional, i.e.  $$ \|f\|_{L^2(G)}^{2+\frac{4}{\mu}} \leq \beta_\mu Q(f)\|f\|_{L^1(G)}^{\frac{4}{\mu}} \quad \hbox{for all } f \in H^1(G) \cap L^1(G),\;f \geq 0,$$ 
where $\mu = 1$ or $\mu=2$, and $\beta_{\mu}$ are constants independent of $f$.

First we mention that it suffices to consider functions $f \in C_c^\infty(G)$, $f \geq 0$, in the above inequalities because every $f \in H^1(G) \cap L^1(G)$, $f \geq 0$, can be simultanuously approximated in the norms $\|\cdot\|_{H^1(G)}$ and $\|\cdot\|_{L^1(G)}$ by functions $f_n \in C_c^\infty(G)$ with $f_n \geq 0$, $n \in \N$ (see \cite{Pang}).

One can prove the one dimensional Nash inequality along the same lines as one proves it in $\R^1$ (see \cite{Haes}), but we have chosen to present another easy proof, more in the spirit of this paper, that reduces the question for the metric graph to that for the real line.
\begin{theo} \label{Nash1}
There exists a constant $\beta_1 >0$ such that for every  $f \in H^1(G) \cap L^1(G)$,  $f \geq 0$, an one dimensional Nash inequality holds
\begin{equation}  \|f\|_{L^2(G)}^6 \leq \beta_1 Q(f)\|f\|_{L^1(G)}^4. \end{equation}  
\end{theo}
\begin{proof} Given any $f \in C_c^\infty(G)$, $f \geq 0$. Consider the subgraph $H$ of $G$ induced by all vertices inside a Ball $B_R(v_0):=\{x \in G:d(x,v_0) <R\}$ with $v_0 \in V$, such that ${\rm supp}(f) \subset B_R(v_0)$, and all vertices not in $B_R(v_0)$ which are adjacent to a vertex inside $B_R(v_0)$. The first set of vertices will be  denoted by $V_{in}$ and the latter by $V_{out}$, clearly $f(v)=0$ for every $v \in V_{out}$. \\
First, we identify every vertex in $V_{out}$ with one vertex $v_{out}$ and obtain a new metric graph $\overline{H}$, possibly having multiple edges between $v_{out}$ and some vertices in $V_{in}$, which we keep, and some loops at $v_{out}$, which we can discard since $f \equiv 0$ there. Now, for every finite graph $\overline{H}$ it is possible to construct  a new graph $\tilde{H}$ out of it by adding at most one new edge between in $\overline{H}$ already adjacent vertices, such that every vertex in the resulting graph $\tilde{H}$ has even degree (because the number of vertices with odd degree must be even, see \cite{Dies}). \\
If $v_1 \sim e \sim v_2$ in $\overline{H}$ and $\tilde{e}$ is a new edge between $v_1$ and $v_2$ in $\tilde{H}$, we choose $\tilde{e}$ to have the same length as $e$ in the metric graph $\tilde{H}$ and extent $f$ to a function  $\tilde{f}$ on $\tilde{H}$ by setting $\tilde{f}(\tilde{e})=f(e)$.\\
$\tilde{H}$ admits an Euler tour (see \cite{Dies}), starting and terminating in $v_{out}$, and the function $\tilde{f}:\tilde{H} \mapsto \R$ can be regarded as a function on the interval $(0,l)$, where the interval $(0,l)$ is obtained by concatenating the edges in their order of appearance during the euler tour. $\tilde{f} \in H^1_0(0,l)$ and hence the Nash inequality for functions in $H^1(\R)$ is valid for $\tilde{f}$. We get
$$\|f\|_{L^2(G)}^6 \leq \|\tilde{f}\|^6_{L^2(\R)} \leq \alpha_1 \int_\R (\tilde{f}'(x))^2\,dx \|\tilde{f}\|^4_{L^1(\R)}\leq 2^5 \alpha_1 Q(f)\|f\|^4_{L^1(G)},$$
where we have used that every edge of $G$ occurs at most twice in the Euler tour.
\end{proof} 
\begin{rem}
The same proof also works for an arbitrary non-compact, connected, locally finite, metric graph with edge lengths uniformly bounded away from $0$. \\
The resulting constant $\beta_1$ does not depend on the tiling or metric graph. The optimal value for $\alpha_1$ was determined in \cite{CL}, $\alpha_1 \sim 0.171$, so $\beta_1 =  6$ will do.
\end{rem} 
In order to achieve a two dimensional Nash inequality, we follow the approach in \cite{Pang}. We take a function $f \in C_c^\infty(G_\Ti)$, $f \geq 0$, and consider its restriction $f_{|\partial P}$ to the boundary of one Polygon $P \in \Ti$. Next we construct a function $F \in H^1({\rm Int}(P)) \cap C(P)$, such that $F_{|\partial P}=f_{|\partial P}$. In the following we simply write $f$ for $f_{|\partial P}$. Furthermore, we need $F$ to fulfil the estimates 
\begin{gather*} \|f\|^2_{L^2(\partial P)} \leq C_1 \|F\|^2_{L^2(P)}, \quad \|F\|_{L^1(P)} \leq C_2 \|f\|_{L^1(\partial P)} \\ \hbox{and} \quad \int_P |\nabla F(x)|^2 dx \leq C_3 \int_{\partial P} (f'(x))dx \end{gather*}
with constants $C_1,C_2,C_3$  depending only on the polygon $P$.\\
Let us say $v_1,\ldots, v_n$ are the vertices of $P$, ordered in that way  anticlockwise along the boundary $\partial P$. We take the incentre $x_P$ of the incircle belonging to $P$, decompose $P$  into $n$ triangles $\Delta_1=v_1v_2x_p$, $\ldots$,  $\Delta_n=v_nv_1x_p$ and consider at first every triangle $\Delta_i$, $i=1,\ldots,n$ separately.\\
$r_P$ is the length of the altitude of every triangle $\Delta_i$ with respect to the base side $v_iv_{i+1}\; ({\rm mod} \, n)$. We denote by $w_i$  the foot points and by $d_i$ the distance between $v_i$ and $w_i$. Without loss of generality, we may assume that $v_i=(0,0)$, $v_{i+1}=(l_i,0)$ and $x_P=(d_i,r_p)$, where $l_i$ is the length of the edge $e_i$ connecting $v_i$ and $v_{i+1}$. Thus $f_{|e_i}$ is regarded as a function on the interval $[0,l_i]$ and is denoted by $f_i$.

\begin{figure}[b]
\setlength{\unitlength}{1cm}
\begin{picture}(8,6)
\put(2,4.5){$P$}
\put(3.5,5){\circle*{0,1}}
\put(3.5,5.1){$v_5$}
\put(3.5,5){\line(1,0){3}}
\put(6.5,5){\circle*{0,1}}
\put(6.5,5.1){$v_4$}
\put(6.5,5){\line(3,-4){1.5}}
\put(8,3){\circle*{0,1}}
\put(8.1,3){$v_3$}
\put(8,3){\line(-3,-4){1.5}}
\put(6.5,1){\circle*{0,1}}
\put(6.6,0.7){$v_2$}
\put(6.5,1){\line(-1,0){3}}
\put(3.5,1){\circle*{0,1}}
\put(3.2,0.7){$v_1$}
\put(3.5,1){\line(-3,4){1.5}}
\put(2,3){\circle*{0,1}}
\put(1.6,3){$v_6$}
\put(2,3){\line(3,4){1.5}}

\put(5,3){\circle*{0,1}}
\put(5.2,3.1){$x_P$}

\put(3.5,1){\line(3,4){1.5}}
\put(6.5,1){\line(-3,4){1.5}}
\put(8,3){\line(-1,0){3}}
\put(2,3){\line(1,0){3}}
\put(3.5,5){\line(3,-4){1.5}}
\put(6.5,5){\line(-3,-4){1.5}}

\put(5,3){\line(0,-1){2}}
\put(5.1,1.8){$r_P$}

\put(5,1){\circle*{0,1}}
\put(5,0.7){$w_1$}
\put(4,0.7){$d_1$}

\put(4.4,1.5){$\Delta_1$}
\put(6.5,2){$\Delta_2$}
\put(6.5,3.7){$\Delta_3$}
\put(5,4){$\Delta_4$}
\put(3.2,3.7){$\Delta_5$}
\put(3.2,2){$\Delta_6$}

\end{picture}
\end{figure}

We parametrize the triangle $\Delta_i$  by 
$$\Delta_i=\{(x,y):0 \leq y \leq r_p,\,\frac{d_i}{r_p}y \leq x \leq l_i-\frac{l_i-d_i}{r_P}y\}$$
and define the function $F_i$ on $\Delta_i$ by
$$F_i(x,y):=\left(1-\frac{y}{r_P}\right)f_i\left(\frac{x-\frac{d_i}{r_P}y}{1-\frac{1}{r_P}y}\right)+ \frac{k}{r_P}y,$$ 
where $k \geq 0$ is a constant not depending on $i$ and to be determined later. Furthermore we set $F:=\cup_{i=1,\ldots,n}F_i:P \mapsto \R$; of course, $F \in H^1({\rm Int}(P)) \cap C(P)$ and $F_{|\partial P}=f_{|\partial P}$.\\
For sake of simplicity we will drop the subscript $i$ during the intermediate steps of the forthcoming analysis and write mostly $r$ for $r_P$.
\begin{align*} \int_{\Delta_i} F_i^2(x,y) dxdy&=\int_0^r \int_{\frac{d}{r}y}^{l-\frac{l-d}{r}y}\left(1-\frac{y}{r}\right)^2 f^2\left(\frac{x-\frac{d}{r}y}{1-\frac{1}{r}y}\right)  \notag \\
               &\hspace{-7mm}+ 2\frac{k}{r}y\left(1-\frac{y}{r}\right) f\left(\frac{x-\frac{d}{r}y}{1-\frac{1}{r}y}\right) +\frac{k^2}{r^2}y^2\, dxdy \notag \\
               &= \int_0^r \int_{\frac{d}{r}y}^{l-\frac{l-d}{r}y}  2\frac{k}{r}y\left(1-\frac{y}{r}\right)  f\left(\frac{x-\frac{d}{r}y}{1-\frac{1}{r}y}\right)  +\frac{k^2}{r^2}y^2\, dxdy \notag \\
                                    &\hspace{-7mm}+\int_0^r\left(1-\frac{y}{r}\right)^3\,dy\int_0^lf^2(x)\,dx \; \geq \;  \frac{r}{4}\|f_i\|^2_{L^2(e_i)}
\end{align*}
Hence, we have easily achieved the first estimate we aim at with $C_1=4/r$
\begin{equation} \label{Ineq1} \|f\|^2_{L^2(\partial P)} \leq \frac{4}{r} \|F\|^2_{L^2(P)}. \end{equation}
The second estimate follows from the equalities below
\begin{align*}
\int_{\Delta_i} F_i(x,y)\,dxdy &=\int_0^r \int_{\frac{d}{r}y}^{l-\frac{l-d}{r}y} \left(1-\frac{y}{r}\right) f\left(\frac{x-\frac{d}{r}y}{1-\frac{1}{r}y}\right) +\frac{k}{r}y\,dxdy \notag \\
                             &=\int_0^r\left(1-\frac{y}{r}\right)^2\,dy\int_0^l f(x)\,dx + \int_0^r l\frac{k}{r}y(1-\frac{y}{r})dy \notag \\
                             &=\frac{r}{3}\|f\|_{L^1(e_i)}+\frac{1}{6}l_ikr 
\end{align*}
if we choose $0\leq k \leq  \|f_{|\partial P}\|_{L^1(\partial P)}/|\partial P|$ and add up for $i=1,\ldots,n$:                           
\begin{equation} \label{Ineq2}\|F\|_{L^1(P)} \leq  \frac{r}{2}\|f\|_{L^1(\partial P)}. \end{equation}
Now we are going to estimate the Dirichlet integral of $F_i$ in terms of the Dirichlet integral of $f_i$. We have
\begin{align*} \frac{\partial F_i}{\partial x}&=f_i'\left(\frac{x-\frac{d_i}{r}y}{1-\frac{1}{r}y}\right),\\
               \frac{\partial F_i}{\partial y}&=\frac{k}{r}-\frac{1}{r}f_i\left(\frac{x-\frac{d_i}{r}y}{1-\frac{1}{r}y}\right)+\frac{x-d_i}{r-y}f_i'\left(\frac{x-\frac{d_i}{r}y}{1-\frac{1}{r}y}\right)
\end{align*}               
and so
\begin{align*} \int_{\Delta_i} \left( \frac{\partial F_i(x,y)}{\partial x}\right)^2+\left( \frac{\partial F_i(x,y)}{\partial y}\right)^2dxdy    &=\int_\Delta \left \{ \left( f'\left(\frac{x-\frac{d}{r}y}{1-\frac{1}{r}y}\right)\right)^2 \right. \\                   
  &\hspace{-40mm}  + \frac{k^2}{r^2} +\frac{1}{r^2}f^2\left(\frac{x-\frac{d}{r}y}{1-\frac{1}{r}y} \right) + \left(\frac{x-d}{r-y}\right)^2 \left(f'\left(\frac{x-\frac{d}{r}y}{1-\frac{1}{r}y}\right)\right)^2 \\ 
  &\hspace{-40mm}+2\frac{k}{r}\left[\frac{x-d}{r-y}f'\left(\frac{x-\frac{d}{r}y}{1-\frac{1}{r}y}\right)-\frac{1}{r}f\left(\frac{x-\frac{d}{r}y}{1-\frac{1}{r}y}\right)\right] \\ 
                  &\hspace{-40mm} \left. - \frac{2}{r}\left(\frac{x-d}{r-y}\right) f\left(\frac{x-\frac{d}{r}y}{1-\frac{1}{r}y}\right) f'\left(\frac{x-\frac{d}{r}y}{1-\frac{1}{r}y}\right) \right \} dxdy.  
\end{align*}
In the fourth summand, $\left(\frac{x-d_i}{r-y}\right)^2 \left(f_i'\left(\frac{x-\frac{d_i}{r}y}{1-\frac{1}{r}y}\right)\right)^2$, on the right hand side we estimate $(\frac{x-d_i}{r-y})^2$ from above by  $(\frac{m_i}{r})^2$ with $m_i={\rm max}\{d_i,|l_i-d_i|\}$. Then, every term can be integrated exactly and we get  
\begin{align} \label{Dir} \int_{\Delta_i} |\nabla F_i(x,y)|^2 \,dxdy & \leq  \left \{ \left( \frac{r}{2}+\frac{m^2_i}{2r} \right ) \| f'_i \|^2_{L^2(e_i)} \right \} \notag \\ 
              &\hspace{-10mm}+\left\{ \frac{l_i k^2}{2 r}-\frac{2k}{r}\|f_i\|_{L^1(e_i)}+\frac{1}{r}\|f_i\|^2_{L^2(e_i)} \right \}  \notag \\
              &\hspace{-10mm} +  \left \{ k \frac{l_i-d_i}{r}f_i(l_i)-\frac{l_i-d_i}{r}f^2_i(l_i)  +k\frac{d_i}{r}f_i(0)- \frac{d_i}{r}f_i^2(0) \right \}. 
\end{align}
Now we take $k \geq 0$ to be the minimum of $\{f_1(0),\ldots,f_n(0),\|f_{|\partial P}\|_{L^1(\partial P)}/|\partial P|\}$.
The last term on the right side of \eqref{Dir} becomes negative and we drop it. We add up for $i=1,\ldots,n$ and arrive at 
\begin{align} \label{Dir2}\int_P |\nabla F(x,y)|^2\,dxdy & \leq \frac {|\partial P| k^2}{2r} -\frac{2k}{r} \|f_{|\partial P}\|_{L^1(\partial P)} \\    &\hspace{-10mm}+\frac{1}{r}\|f_{|\partial P}\|^2_{L^2(\partial P)} +  \left(  \frac{r}{2}+\frac{m^2}{2r} \right ) \| f'_{|\partial P} \|^2_{L^2(\partial P)}, \notag
\end{align}                           
where $m \leq |\partial P|/2$ is the maximum of $\{m_i: i=1, \ldots,n\}$.\\
As the value $k$ is taken by the function $f_{|\partial P}$, say at point $z$, and setting $g:=f_{|\partial P}-k$, we obtain, by splitting $\partial P$ at the point $z$ into two arcs, $(\partial P)_1$ and $(\partial P)_2$, of equal length and writing $|x|$ for the arclength from $z$ to $x$,
\begin{align*}
\int_{\partial P} g^2(x)\,dx &= \int_{(\partial P)_1} g^2(x)\,dx +\int_{(\partial P)_2} g^2(y)\,dy\\ &=\int_{(\partial P)_1}\left (\int_z^x g'(t)\,dt \right)^2\,dx+\int_{(\partial P)_2} \left (\int_z^y g'(t)\,dt \right)^2\,dy\\
&\leq \int_{(\partial P)_1}|x| \left (\int_z^x (g'(t))^2\,dt \right )\,dx+ \int_{(\partial P)_2}|y| \left ( \int_z^y (g'(t))^2\,dt \right)\,dy\\
&\leq \left ( \int_{(\partial P)_1} (g'(t))^2\,dt \right )\int_{(\partial P)_1}|x|\,dx+ \left ( \int_{(\partial P)_2} (g'(t))^2\,dt \right )\int_{(\partial P)_2}|y|\,dy\\
&=\left (\int_{(\partial P)_1} (g'(t))^2\,dt + \int_{(\partial P)_2} (g'(t))^2\,dt\right )\int_0^{\partial  P/2}x\,dx\\
&= \frac{1}{2}\left(\frac{|\partial P|}{2} \right)^2 \left (\int_{\partial P}(g'(t))^2\,dt \right ) = \frac{1}{8}|\partial P|^2 \left (\int_{\partial P}(g'(t))^2\,dt \right ),
\end{align*}
and therewith
\begin{align*} \frac{|\partial P|^2}{8}\|f_{|\partial P}'\|^2_{L^2(\partial P)} &= \frac{|\partial P|^2}{8}\|(f_{|\partial P}-k)'\|^2_{L^2(\partial P)}= \frac{|\partial P|^2}{8}\|g'\|^2_{L^2(\partial P)}\geq \|g\|^2_{L^2(\partial P)} \\
                                            & \hspace{-7mm} = \int_{\partial P} (f_{|\partial P}(x)-k)^2 \,dx=\|f_{|\partial P}\|^2_{L^2(\partial P)} -2k \|f_{|\partial P}\|_{L^1}(\partial P) + k^2|\partial P|, 
\end{align*}  
or 
\begin{align*}
\frac{1}{r}\|f_{|\partial P}\|^2_{L^2(\partial P)} \leq \frac{ |\partial P|^2}{8r} \|f_{\partial P}'\|^2_{L^2(\partial P)}+\frac{2k}{r} \|f_{|\partial P}\|_{L^1(\partial P)}-\frac{ |\partial P| k^2}{r}.
\end{align*}                                         
We combine this with \eqref{Dir2}, and the last of our desired inequalities follows 
\begin{equation} \label{Ineq3} \int_P |\nabla F(x,y)|^2\,dxdy \leq  \left(  \frac{r}{2} + \frac{ |\partial P|^2}{4r} \right ) \| f'_{|\partial P} \|^2_{L^2(\partial P)}. \end{equation} 
\begin{theo} \label{Nash2}
There exists a constant $\beta_2 >0$ such that for every  $f \in H^1(G) \cap L^1(G)$,  $f \geq 0$, a two dimensional Nash   inequality  holds
\begin{equation} \|f\|_{L^2(G)}^4 \leq \beta_2 Q(f)\|f\|_{L^1(G)}^2. \end{equation} 
\end{theo}       
\begin{proof}
Given $f \in C_c^\infty(G_\Ti)$, $f \geq 0$, we construct in the interior of every Polygon $P \in \Ti$ a function $F_P$  as above and glue them together to get a Function $F \in H^{1}(\R^2)$. Now we can apply the Nash inequality for functions in $H^1(\R^2)$ to $F$ and obtain, using inequalities \eqref{Ineq1}, \eqref{Ineq2}, and \eqref{Ineq3}, 
\begin{align*} \|f\|_{L^2(G)}^4 &=\frac{1}{4}\left(2 \sum_{e \in E_\Ti}\int_e f_{|e}^2\right )^2 =\frac{1}{4}\left(\sum_{P \in \Ti}\int_{\partial P} f^2_{|\partial P} \right )^2 \\
&\leq \frac{1}{4}\left (\sum_{P \in \Ti} \frac{4}{r_P}\int_{P} F^2_{|P} \right )^2 \leq \frac{4}{h^2} \left (\int_{\R^2} F^2 \right)^2 = \frac{4}{h^2} \|F\|_{L^2(\R^2)}^4 \\
 &\leq \alpha_2 \frac{4}{h^2} \left(\int_{\R^2} |\nabla F|^2 \right)\|F\|^2_{L^1(\R^2)} =\alpha_2 \frac{4}{h^2}\left(\sum_{P \in \Ti}\int_P |\nabla F|^2 \right)\left(\sum_{P \in \Ti}\int_P  |F| \right)^2 \\
 &\leq \alpha_2 \frac{4}{h^2}\left( \sum_{P \in \Ti} \left( \frac{r_P}{2}+\frac{| \partial P|^2}{4r_P} \right ) \int_{\partial P} (f')^2 \right )\left(\sum_{P \in \Ti} \frac{r_P}{2} \int_{\partial P} |f| \right)^2\\
 &\leq \alpha_2 \frac{4}{h^2}\left(\frac{H}{2}+\frac{M^2}{4h} \right )\left( 2 \int_G (f')^2\right ) \left(\frac{H}{2}2\int_G |f| \right)^2\\
                         &\leq  4 \alpha_2 \frac{H^2}{h^2}\left(H +\frac{M^2}{2h}\right)Q(f)\|f\|^2_{L^1(G)},
\end{align*}
where $H \geq r_P$, $h \leq r_P$, and $M \geq |\partial P|$ for every $P \in \Ti$ are as explained in the introduction.
\end{proof}  
\begin{rem}
This time the constant $\beta_2$ depends on the specific tiling via the parameters $H_\Ti$, $h_\Ti$, and $M_\Ti$. The optimal value for $\alpha_2$ was determined in \cite{CL}, $\alpha_2 \sim 0.087$, so $\beta_2=\frac{H_\Ti^2}{2h_\Ti^2}\left(H_\Ti +\frac{M_\Ti^2}{2h_\Ti}\right)$ will do.
\end{rem}
 
\section{Kernel Estimates}  
There is a well known equivalence between Nash inequalities for Dirichlet forms and ultracontractive estimates for the corresponding symmetric Markov semigroups (see \cite{Da1}, Theorem 2.4.6.). Hence, we have the following corollary of Theorems \ref{Nash1} and \ref{Nash2}.
\begin{theo} \label{ultra}
For every $f \in L^1(G) \cap L^\infty(G)$ and every $t >0$ the following ultracontractive estimates hold true
\begin{equation} \label{ultra1}
\|e^{-At}f\|_{L^\infty(G)} \leq \gamma_\mu t^{-\frac{\mu}{2}}\|f\|_{L^1(G)}, \quad \mu=1 \hbox{ or } \mu=2,
\end{equation}
with $\gamma_1=(\beta_1/2)^{1/2}$  and $\gamma_2=\beta_2$.
\end{theo}
\begin{rem} \label{ultrarem}
For small times  the one dimensional estimate ($\mu=1$) is sharper, but in the long run the two dimensional estimate ($\mu=2$) gives the better bound.  The exact time the two dimensional behaviour starts to dominate depends on the constant $\beta_2$ and therefore on the values of $H_\Ti$, $h_\Ti$ and $M_\Ti$. We can rewrite \eqref{ultra1} to make this kind of ``dimension transition'' more explicit 
\begin{equation} \label{ultra2}
    \|e^{-At}\|_{L^1 \to L^\infty} \leq \left \{ \begin{array}{ll} \sqrt{3}\, t^{-1/2} &  \quad (0 < t \leq \frac{1}{3} \beta^2), \\
                                                          \beta \, t^{-1}   &  \quad   (t \geq \frac{1}{3} \beta^2), \end{array} \right .
\end{equation}
where we have set $\beta_1=6$ and $\beta=\beta_2$. 
\end{rem}
We know that the semigroup $e^{-At}$ has a kernel $k(t,\cdot,\cdot) \in L^\infty(G \times G)$, $t >0$, so we would like to establish Gaussian estimates. We refer to   \cite{DP} for the following result, which indeed gives us the desired heat kernel bound (see also \cite{Da1} and \cite{Pang}). 
\begin{theo} \label{gauss}
There exists a constant $\eta > 0$ such that for all $t> 0$ and $x,y \in G$ the kernel estimate
$$ 0 \leq k(t,x,y) \leq a(t,d(x,y))e^{-d^2(x,y)/(4t)} $$ 
holds with $a(t,d(x,y))= \eta \min \{ t^{-1/2}(1+d^2(x,y)/t)^{1/2}, t^{-1}(1+d^2(x,y)/t)\}$.
\end{theo}

\section{Concluding Remarks}
It might be interesting to investigate around which time the global, i.e. the two dimensional nature of a concrete metric graph, arising from a tiling of the plane, begins to dominate its local, i.e. one dimensional nature. We have seen in Remark \ref{ultrarem} that this essentially depends on $\beta=\beta_2$, for which we can set $\beta=\frac{H_\Ti^2}{2h_\Ti^2}\left(H_\Ti +\frac{M_\Ti^2}{2h_\Ti}\right)$, and therefore on the parameters $ H_\Ti$, $h_\Ti$, and $M_\Ti$. Of course, the estimates are not necessarily sharp and the time $\frac{1}{3}\beta^2$ appearing in \eqref{ultra2} is hence a bit arbitrary.  Nevertheless, it gives a hint how things change in dependence of the tiling.\\
In particular, one can see that the ``transition time'' grows quadratically if we dilate the tiling linearly.

As the Nash inequalities remain valid for forms $Q_b(f,g):=Q(f,g)+b(f_{|V},g_{|V})$, where $f_{|V}$ denotes the restriction of $f$ to the vertex set and $b(\cdot,\cdot)$ be a bounded, symmetric  Dirichlet form on $l^2(V)$, all the above conclusions also hold for the semigroups generated by the operators associated with these forms, and especially for Robin type vertex conditions $b_v f(v)+\sum_{e \sim v}\frac{\partial f_{|e}}{\partial n}(v)=0$ with $0 \leq b_v \leq M < \infty$ for all $v\in V$. Furthermore, one can substitute $\int_{G} f'(x)g'(x) d\mu(x)$ by $\int_{G} \alpha(x) f'(x)g'(x) d\mu(x)$ with $\alpha(\cdot) \in L^\infty(G)$ and $\alpha(x) \geq \alpha >0$.

\end{document}